\documentclass[onecolumn, a4paper, 12pt, conference]{ieeeconf}      
\IEEEoverridecommandlockouts                              
\let\labelindent\relax

\usepackage[utf8]{inputenc}
\usepackage{graphics} 
\usepackage{amsmath} 
\usepackage{amssymb}  
\usepackage{amsfonts}
\usepackage{comment}
\usepackage{enumitem}
\usepackage{cases,xspace,enumitem}

\usepackage{version}
\includeversion{tac}
\excludeversion{arxiv}

\usepackage{tikz}
\usetikzlibrary{positioning}

\usepackage[prependcaption,colorinlistoftodos]{todonotes}

\definecolor{gnred}{RGB}{255,91,89}

\definecolor{gnred1}{RGB}{71,0,0} 
\definecolor{gnred2}{RGB}{117,0,0} 
\definecolor{gnred3}{RGB}{164,0,0} 
\definecolor{gnred4}{RGB}{211,0,0} 
\definecolor{gnred5}{RGB}{255,0,0} 
\definecolor{gnred6}{RGB}{255,42,34} 
\definecolor{gnred7}{RGB}{255,91,89} 

\definecolor{gnblue1}{RGB}{0,36,71}   
\definecolor{gnblue2}{RGB}{0,60,118}  
\definecolor{gnblue3}{RGB}{0,85,164}
\definecolor{gnblue4}{RGB}{0,108,212}
\definecolor{gnblue4}{RGB}{0,108,212}
\definecolor{gnblue5}{RGB}{0,133,255}  
\definecolor{gnblue6}{RGB}{35,156,255} 
\definecolor{gnblue7}{RGB}{88,177,255} 

\definecolor{gnbrown1}{RGB}{71,27,0}  
\definecolor{gnbrown2}{RGB}{117,45,0} 
\definecolor{gnbrown3}{RGB}{164,62,0} 
\definecolor{gnbrown4}{RGB}{211,80,0} 
\definecolor{gnbrown5}{RGB}{255,97,0} 
\definecolor{gnbrown6}{RGB}{255,127,26} 
\definecolor{gnbrown7}{RGB}{255,155,86} 

\renewcommand{\theenumi}{(\roman{enumi})}

\newcommand\Item[1][]{%
  \ifx\relax#1\relax  \item \else \item[#1] \fi
  \abovedisplayskip=0pt\abovedisplayshortskip=0pt~\vspace*{-\baselineskip}}

\usepackage{hyperref}
\hypersetup{
    colorlinks=true,
    linkcolor=blue,
    filecolor=magenta,      
    urlcolor=cyan,
    pdftitle={Overleaf Example},
    pdfpagemode=FullScreen,
    }

\usepackage{amsthm}

\newtheorem{definition}{Definition}
\newtheorem{thm}{Theorem}
\newtheorem{lem}[thm]{Lemma}
\newtheorem{pro}[thm]{Proposition}
\newtheorem{cor}[thm]{Corollary}
\newtheorem{rem}[thm]{Remark}

\newtheorem{assumption}{Assumption}
\newcommand{\bd}{\begin{definition}} 
\newcommand{\ed}{\end{definition}} 
\newcommand{\bp}{\begin{pro}} 
\newcommand{\ep}{\end{pro}} 
\newcommand{\bt}{\begin{thm}} 
\newcommand{\et}{\end{thm}}
\newcommand{\blm}{\begin{lem}} 
\newcommand{\elm}{\end{lem}}
\newcommand{\bi}{\begin{itemize}} 
\newcommand{\ei}{\end{itemize}} 
\newcommand{\bds}{\begin{description}} 
\newcommand{\eds}{\end{description}} 
\newcommand{\beq}{\begin{equation}} 
\newcommand{\eeq}{\end{equation}} 
\let\eps\varepsilon

\newcommand{\R}{\mathbb{R}}
\newcommand{\K}{\mathcal{K}}
\renewcommand{\S}{\bold{S}}

\DeclareMathOperator{\Img}{\operatorname{Im}}
\newcommand{\spec}{\operatorname{spec}}
\newcommand{\spn}{\operatorname{\mathsf{span}}}
\newcommand{\abs}[1]{\left|#1\right|}

\newcommand{\norm}[1]{\|#1\|}

\newcommand{\wlognorm}[3]{\mu_{#1,#2}(#3)}
\newcommand{\lognorm}[2]{\mu_{#1}(#2)}

\newcommand{\derp}[2]{\frac{\partial #1}{\partial #2}}

\newcommand{\trasp}[1]{#1^{\textsf{T}}}

\newcommand{\diag}[1]{[#1]}

\newcommand{\until}[1]{\{\,1,\dots, #1\,\}}

\newcommand{\subscr}[2]{#1_{\textup{#2}}}

\newcommand{\setdef}[2]{\{#1 \; | \; #2\}}

\newcommand{\Bigsetdef}[2]{\Big\{#1 \; \big| \; #2\Big\}}
\newcommand{\map}[3]{#1 \colon #2 \rightarrow #3}

\newcommand{\osLip}{\operatorname{\mathsf{osL}}}

\newcommand{\mpi}[1]{{#1}^{\dagger}}

\newcommand{\xstar}{x^{*}}

\newcommand{\hop}{\subscr{x}{H}}
\newcommand{\fr}{\subscr{x}{F}}
\newcommand{\hopdot}{\subscr{\dot x}{H}}
\newcommand{\frdot}{\subscr{\dot x}{F}}
\newcommand{\uh}{\subscr{u}{H}}
\newcommand{\ufr}{\subscr{u}{F}}

\newcommand{\Q}{Q}
\newcommand{\ql}[1]{Q_{\textup{F},#1}}
\newcommand{\qr}[1]{Q_{\textup{H},#1}}

\newcommand{\HNN}{HNN\xspace}
\newcommand{\FNN}{FNN\xspace}

\newcommand{\fh}{\subscr{f}{H}}
\newcommand{\ffr}{\subscr{f}{F}}
\newcommand{\fsplit}{\theta}
\newcommand{\lm}{b}
\newcommand{\lmax}{\alpha(W)}

\newcommand{\xperp}{x_{\perp}}
\newcommand{\xparal}{x_{\parallel}}

\newcommand{\PF}{\subscr{\mathcal{P}}{F}}
\newcommand{\PH}{\subscr{\mathcal{P}}{H}}
\newcommand{\nperp}{n_{\perp}}
\newcommand{\nparal}{n_{\parallel}}
\newcommand{\Uperp}{U_{\perp}}
\newcommand{\Uparal}{U_{\parallel}}

\newcommand{\sat}[2]{\operatorname{sat}_{#1}({#2})}

\newcommand{\1}
{\mbox{\fontencoding{U}\fontfamily{bbold}\selectfont1}}
\newcommand{\0}{\mbox{\fontencoding{U}\fontfamily{bbold}\selectfont0}}

\usepackage{titlesec}

\titlespacing*{\section}
{0pt}{3ex plus 1ex minus .2ex}{2ex plus .2ex}
\titlespacing*{\subsection}
{0pt}{3ex plus 1ex minus .2ex}{2ex plus .2ex}
\titlespacing*{\subsubsection}
{0pt}{1.5ex plus 1ex minus .2ex}{1ex plus .2ex}
\titlespacing*{\paragraph}
{0pt}{1.5ex plus 1ex minus .2ex}{1ex plus .2ex}

\title{Euclidean Contractivity of Neural Networks\\
with Symmetric Weights}

\author{Veronica Centorrino$^a$, Anand Gokhale$^b$, Alexander Davydov$^b$,\\
Giovanni Russo$^c$ and Francesco Bullo$^b$
  \thanks{This work was in part supported by AFOSR project FA9550-21-1-0203. The authors
    thank Dr.\ Leo Kozachkov for insightful comments.}%
\thanks{$^a$Veronica Centorrino is with Scuola Superiore Meridionale, University of Naples Federico II, Italy. {\tt\small veronica.centorrino@unina.it.}}
\thanks{$^b$Anand Gokhale, Alexander Davydov, and Francesco Bullo are with the Center for Control, Dynamical 
Systems, and Computation, UC Santa Barbara, Santa Barbara, CA 93106 USA. {\tt\small 
anand\_gokhale@ucsb.edu, davydov@ucsb.edu, bullo@ucsb.edu}.}
\thanks{$^c$Giovanni Russo is with the Department of Information and Electric Engineering and Applied Mathematics, University of Salerno, Italy. {\tt\small giovarusso@unisa.it.}}
}

\date{}
\begin{document}
\maketitle

\begin{abstract}
\normalsize
This paper investigates stability conditions of continuous-time Hopfield and firing-rate neural networks by leveraging contraction theory.
First, we present a number of useful general algebraic results on matrix polytopes and products of symmetric matrices. Then, we give sufficient conditions for strong and weak Euclidean contractivity, i.e.,  contractivity with respect to the $\ell_2$ norm, of both models with symmetric weights and (possibly) non-smooth activation functions. Our contraction analysis leads to contraction rates which are log-optimal in almost all symmetric synaptic matrices.
Finally, we use our results to propose a firing-rate neural network model to solve a quadratic optimization problem with box constraints.
\end{abstract}

\section{Introduction}
Continuous-time recurrent neural networks (RNNs) are dynamical models
widely studied in computational neuroscience and machine learning.  Recent
interest has focused on establishing the contractivity properties of RNNs.
Contracting dynamics are robustly stable, feature computationally friendly
methods for equilibrium computation, and enjoy many other properties.
Motivated by optimization~\cite{DWT-JJH:84,AB-TRP:93} and neuroscientific
applications~\cite{CJR-DHJ-RGB-BAO:08},~\cite[Chapter~17]{WG-WMK-RN-LP:14}, this paper focuses on symmetric
synaptic interactions.

While a comprehensive contractivity analysis with respect to $\ell_1$ and
$\ell_\infty$ norms was recently presented in~\cite{AD-AVP-FB:22q}, the
corresponding analysis with respect to weighted Euclidean norms is not
complete yet.  A recent breakthrough in this direction was obtained
by~\cite{LK-ME-JJES:22}; this work extends and complements these results (a
detailed comparison is offered below).

Two common models of RNNs are the \emph{firing-rate neural network} (\FNN)
and \emph{Hopfield neural network} (\HNN); the main difference being the
order by which the activation function acts.  Under mild assumptions, FNNs
are positive systems and, arguably, more biologically-plausible. HNNs are
relevant in optimization and machine learning~\cite{DWT-JJH:84,AB-TRP:93,
  CJR-DHJ-RGB-BAO:08,AY-JX-MR-ER:22}. For certain synaptic matrices and initial
conditions, \FNN and \HNN are known to be equivalent via an appropriate
change of coordinates and input transformation~\cite{KDM-FF:12}.
However, the understanding of this partial correspondence is not complete
and, as we will show below, their contractivity properties are not exactly
coincident.

\paragraph{Related literature}
RNNs naturally emerge when modelling neural
processes~\cite{WG-WMK-RN-LP:14}. Critical questions when studying RNNs are
related to finding conditions that guarantee stability and robustness of
the network.
For example, sufficient conditions for the stability of HNNs are given in~\cite{MF-AT:95} based on the use of Lyapunov diagonally stable matrices.
Stability and robustness can be simultaneously established using contraction theory.
Indeed, contracting systems exhibit highly ordered
transient and asymptotic behaviors that appear to be convenient in the
context of RNNs. For example: (i) initial conditions are exponentially
forgotten~\cite{WL-JJES:98}; (ii) time-invariant dynamics admits a unique
globally exponential stable equilibrium~\cite{WL-JJES:98}; (iii)
contraction ensures entrainment to periodic inputs~\cite{GR-MDB-EDS:10a}
and (iv) enjoy highly robust behavior, such as input-to-state
stability~\cite{SX-GR-RHM:21}. (v) Moreover, efficient numerical algorithms
can be devised for numerical integration and fixed point computation of
contracting systems~\cite{SJ-AD-AVP-FB:21f}.  Recently, non-Euclidean
contractivity of RNNs is studied in~\cite{AD-AVP-FB:22q} and
in~\cite{VC-FB-GR:22k}, where stability properties of \HNN and \FNN with
dynamic synapses undergoing Hebbian learning are proposed.  Euclidean
contractivity is studied in~\cite{LK-ML-JJES-EKM:20} to analyze the
stability of RNNs with dynamic synapses and in~\cite{LK-ME-JJES:22}, where
a number of contractivity conditions are proposed. Finally, the design of
norms minimizing the logarithmic norm is reviewed
in~\cite[Section~2.7]{FB:23-CTDS}.
\paragraph{Contributions:} our main results are a set of sufficient conditions characterizing strong and weak infinitesimal contractivity properties (see Section~\ref{sec:math_preliminaries} for the definitions) of FNNs and HNNs with symmetric weights and possibly non-smooth activation functions. We also establish a lower bound on the contraction rate and, remarkably, demonstrate that the bound is log-optimal in almost all symmetric weight matrices. One of the main benefits of our approach to the study of FNNs and HNNs is that, with just a single condition, it ensures global exponential convergence, along with all the other useful properties of contracting systems. The main results leverage a number of general algebraic results, which are interesting {\em per se} and are also a contribution of this paper. With these algebraic results, we: (i) determine a weighted $\ell_2$ norm for matrix polytopes which is log-optimal for almost all synaptic matrices; (ii) give a lower bound on the spectral abscissa of matrix polytopes; (iii) provide optimal and log-optimal norms for the product of symmetric matrices. Finally, we leverage our sufficient conditions for contractivity to propose a \FNN solving certain quadratic optimization problems with box constraints. 

Our results for strong infinitesimal contractivity of the \FNN and \HNN models with symmetric weights are based on and generalize~\cite[Theorem 2]{LK-ME-JJES:22}.
Specifically, (i) we provide the explicit expression of the matrix weights for which the models are contracting. The matrices we find are different for the two models, highlighting the importance of choosing the appropriate model based on the properties being studied;
(ii) we address the weak contractivity case, i.e., when the contraction rate is $0$, making it applicable for, e.g., systems that enjoy conservation or invariance properties;
(iii) we handle weakly increasing and (iv) locally Lipschitz activation functions, allowing us to consider common activation functions such as the rectified linear unit (ReLU) and soft thresholding functions.

\section{Mathematical Preliminaries}\label{sec:math_preliminaries}
We denote by $\map{(\cdot)_+}{\R}{\R_{\geq0}}$ the function $(z)_+ = z$ if $z >0$, $(z)_+ = 0$ if $z \leq 0$.
Given $x \in \R^n$, we define $\diag{x} \in \R^{n \times n}$ to be the diagonal matrix with diagonal entries equal to $x$. Vector inequalities of the form $x \leq (\geq)~y$ are entrywise. We let $\1_n$, $\0_n \in \R^n$ be the all-ones and all-zeros vectors, respectively, $I_n$ be the $n \times n$ identity matrix, and $\S^n$ be the set of real symmetric $n\times n$ matrices.
For $A\in \R^{n \times n}$, let $\spec(A)$,  $\rho(A) := \max \setdef{\abs{\lambda}}{\lambda \in \spec(A)}$ and $\alpha(A) := \max \setdef{\Re(\lambda)}{\lambda \in \spec(A)}$ denote the spectrum, {spectral radius} and the {spectral abscissa} of $A$, respectively; here $\Re(\lambda)$ denotes the real part of $\lambda$.
For $A\in\S^n$, let $\subscr{\lambda}{min}(A)$ and $\subscr{\lambda}{max}(A)$ denote its minimum and maximum eigenvalue, respectively.
Given $A, B \in \S^n$, we write $A \preceq B$ (resp. $A \prec B$) if $B-A$ is positive semidefinite (resp. definite).
The Moore–Penrose inverse of $A\in \R^{n\times n}$ is the unique matrix $\mpi{A} \in \R^{n\times n}$ such that $A\mpi{A}A = A$, $\mpi{A}A\mpi{A} = \mpi{A}$, with $A\mpi{A}$, $\mpi{A}A \in \S^n$.
Finally, whenever it is clear from the context, we omit to specify the dependence of functions on time $t$.
\subsection{Norms and induced norms}
Let $\| \cdot \|$ denote both a norm on $\R^n$ and its corresponding induced matrix norm on $\R^{n \times n}$. Given $A \in \R^{n \times n}$ the \emph{logarithmic norm} (log-norm) induced by $\| \cdot \|$ is 
\[
\mu(A) := \lim_{h\to 0^{+}} \frac{\norm{I_n{+}h A} -1}{h}.
\]
Specifically, the Euclidean vector norm, matrix norm, and log-norm are, respectively:
${\displaystyle \norm{x}_2 = \sqrt{x^\top x}}$, $\allowbreak \displaystyle {\norm{A}_2 = \sqrt{\subscr{\lambda}{max}(A^\top A)}}$, and $\displaystyle \lognorm{2}{A} = \frac{1}{2}\subscr{\lambda}{max}\left(A{+}A^\top\right)$.

For an $\ell_p$ norm, $p \in [1,\infty]$, and for an invertible matrix $\Q \in \R^{n\times n}$, the $\Q$-weighted $\ell_p$ norm is defined as $\norm{x}_{p,\Q} := \norm{\Q x}_p$. The corresponding log-norm is $\wlognorm{p}{\Q}{A} =\lognorm{p}{\Q A \Q^{-1}}$.
Specifically, the weighted Euclidean vector norm, matrix norm, and log-norm are, respectively: $\displaystyle  \norm{x}_{2,Q} = \norm{Qx}_2$, $\displaystyle  \norm{A}_{2,Q^{1/2}} =  \sqrt{\subscr{\lambda}{max}(Q^{-1}A^\top Q A)}$, and $\displaystyle {\wlognorm{2}{Q^{1/2}}{A} = \frac{1}{2}\subscr{\lambda}{max}\left(QAQ^{-1}{+}A^\top\right)}$.

For two invertible matrices $\Q_1$, $\Q_2 \in \R^{n \times n}$, it holds
\begin{equation} \label{prop:silly-weightedmu}
\wlognorm{p}{\Q_1\Q_2}{A}=
\wlognorm{p}{\Q_1}{\Q_2A\Q_2^{-1}}.
\end{equation}

Given $\map{f}{\R_{\geq 0} \times C}{\R^n}$, with $C \subseteq \R^n$ open and connected, we denote by $\osLip(f_t)$ the \emph{one-sided Lipschitz constant} of $f_t:=f(t,\cdot)$.
For continuously differentiable $f_t$ and convex set $C$ it holds $$\osLip(f_t) = \sup_{x\in C}\lognorm{}{Df(t,x)},$$ where $Df(t,x) := \partial f(t,x)/\partial x$ is the Jacobian of $f$ with respect to $x$. We write $\osLip_{p,Q}(f_t)$ to specify that the one-sided Lipschitz constant is computed with respect to a $\Q$-weighted $\ell_p$ norm.
Specifically, for the weighted Euclidean norm we have:
$$\osLip_{2,Q^{1/2}}(f_t) = \sup_{x,y\in C,x\neq y} \frac{(x-y)^\top Q(f(x) -f(y))}{\|x-y\|_{2,Q^{1/2}}^2}.$$

We refer to~\cite{FB:23-CTDS} for a recent review of those tools.
\subsection{Contraction theory for dynamical systems} We start with the following
\bd \label{def:contracting_system}
Given a norm, a function $\map{f}{\R_{\geq 0} \times C}{\R^n}$, with $C \subseteq \R^n$ $f$-invariant, open and convex, and a constant $c >0$ ($c = 0)$ referred as \emph{contraction rate}, $f$ is strongly (weakly) infinitesimally contracting on $C$ if
\[
\osLip(f_t) \leq -c,  \textup{ for all } t\in \R_{\geq 0},
\]
or, equivalently for differentiable vector fields, if
\beq\label{cond:contraction_log_norm}
\mu(Df(t,x)) \leq -c,  \textup{ for all } x \in C  \textup{ and } t\in \R_{\geq0}.
\eeq
\ed
One of the main benefits of contraction theory is that, with just a single condition, it ensures global
exponential convergence, along with other useful properties, as highlighted in the introduction Section.

The next result~\cite[Theorem 16]{AD-AVP-FB:22q} allows using condition~\eqref{cond:contraction_log_norm} for locally Lipschitz function, for which, by Rademacher’s theorem, $Df(t,x)$ exists almost everywhere (a.e.) in $C$.
\bt\label{th:equivalence_loc_lip_func}
Consider a norm, a function $\map{f}{\R_{\geq 0}\times C}{\R^n}$ locally Lipschitz on $C \subset \R^n$ open and convex set. Then for every $c \in \R$ the following statements are equivalent:
\begin{enumerate}
\item $\osLip(f_t) \leq c$,\ for all $t\in \R_{\geq 0}$,
\item $\lognorm{ }{Df(t,x)} \leq c$,\ for a.e. $x \in C$ and $t\in \R_{\geq0}$.
\end{enumerate}
\et
\subsection{Hopfield and firing-rate continuous-time neural networks}
We are interested in the following continuous-time \FNN and \HNN models defined, respectively, as:
\begin{align}
   \frdot &= -\fr{+}\Phi(W\fr+\ufr) := \ffr(\fr,\ufr) \label{eq:firing_rate_nn},\\
   \hopdot &={-}\hop{+}W \Phi(\hop){+}\uh : = \fh(\hop,\uh), \label{eq:hopfield_nn}
\end{align}
where: $\fr$, $\hop \in \R^n$ are neural activation vectors, $\map{\Phi}{\R^n}{\R^n}$ is a nonlinear and diagonal activation function, i.e., for $x \in \R^n$, $(\Phi(x))_i = \phi(x_i)$, where $\map{\phi}{\R}{\R}$. $W \in \R^{n \times n}$ is the synaptic matrix, with $W_{ij} \in \R$ being the synaptic weight from neuron $j$ to neuron $i$. Finally, $\ufr$, $\uh \in \R^n$ are the external stimuli in the \FNN and \HNN, respectively.  The models~\eqref{eq:firing_rate_nn} and~\eqref{eq:hopfield_nn} assume homogeneous dissipation rates; we leave the heterogeneous case to future work.
\begin{rem}
When the activation function is non-negative the positive orthant is forward-invariant for $\ffr$ in~\eqref{eq:firing_rate_nn} and $\fr$ is interpreted as a firing-rate.  Instead, in~\eqref{eq:hopfield_nn} $\hop$ is sign indefinite and is interpreted as a membrane potential.
\end{rem}
\section{Main Results}\label{sec:main_results}
This section presents the main results of the paper.
Namely, we study Euclidean contractivity properties of continuous-time RNNs with symmetric weights.

First, we give algebraic results on weighted $\ell_2$ norms of certain matrix polytopes.
Then, we use those results to give sufficient conditions for the strong infinitesimal contractivity of the \FNN and the \HNN with symmetric weights with respect to weighted Euclidean norms.

\begin{assumption}[Symmetric synaptic weights]
  \label{ass:symmetric_synaptic_matrix}
  The synaptic matrix $W\in \R^{n\times n}$ is symmetric. 
\end{assumption}

Under Assumption~\ref{ass:symmetric_synaptic_matrix}, the eigenvalues of
$W$ are real, $\alpha(W) = \subscr{\lambda}{max}(W)$ and $W \preceq \lmax
I_n$. Moreover, $W$ can be decomposed as \beq \label{eq:decomposition_W}
W=U\Lambda U^\top, \eeq where $U\in \R^{n\times n}$ is the orthogonal
matrix whose columns are the eigenvectors of $W$, and $\Lambda = [\lambda]
\in \R^{n\times n}$ is diagonal with $\lambda \in \R^n$ being the vector of
the eigenvalues of $W$.

Given $\lm > 0$, we define ${\map{\fsplit_{\lm}}{]{-}\infty,\lm]}{[2\lm,+\infty[}}$ by
\begin{equation}\label{eq:theta}
\fsplit_{\lm}(z):=2\lm\big(1+\sqrt{1-z/\lm}\big), \quad \forall z \in {]{-}\infty,\lm]}.
\end{equation}
We illustrate $\fsplit_{\lm}(\cdot)$ in Figure~\ref{fig:plot_theta_b}. For our derivations, it is useful to introduce the shorthand notation ${\fsplit_{\lm}(\Lambda):=\diag{(\fsplit_{\lm}(\lambda_1),\dots,\fsplit_{\lm}(\lambda_n))}}$. Also, we introduce $\ql{\lm}\in \R^{n\times n}$
\begin{align}
\ql{\lm} & := U \fsplit_{\lm}(\Lambda) U^\top\succ 0, \label{eq:ql}
\end{align}
and,  when $W$ is invertible, $\qr{\lm} \in \R^{n\times n}$ is defined as
\begin{align}
\qr{\lm} & := \ql{\lm} W^{-1} = U \fsplit_{\lm}(\Lambda)\Lambda^{-1} U^\top
\succ 0.\label{eq:qr}
\end{align}
\begin{figure}[!ht]
\scalebox{0.55}{\input{plot_theta_b.pgf}}
\centering
\caption{Plot of the function $\fsplit_{\lm}(\cdot)$ with $\lm = 5$.}
\label{fig:plot_theta_b}
\end{figure}
\begin{rem}
The matrix $\qr{\lm}$ defined in~\eqref{eq:qr} can be written as $\qr{\lm}= U g_{\lm}(\Lambda)\trasp{U}$, where we use the notation $g_{\lm}(\Lambda) := \diag{g_{\lm}(\lambda_1), \dots, g_{\lm}(\lambda_n)}$, with $g_{\lm}(\cdot)$ defined by
\beq
g_{\lm}(z) := 2b\frac{1+\sqrt{1-z/\lm}}{z}, \quad \forall z \in {]-\infty,b] \setminus\{0\}}.
\eeq
\end{rem}
\subsection{Results on the
Euclidean log-norm of matrix polytopes}
First, we give the following definition for polytopes.
\bd[Log-optimal and log-$\eps$-optimal norms for matrix polytopes]
\label{def:log-opt-polytopes}
Given $A_1, \dots, A_m \in \R^{n \times n}$, consider the polytope $$\mathcal{P} = \Bigsetdef{\sum_{j = 1}^{m} \beta_j A_j}{\beta_j \geq 0, \sum_{j = 1}^{m} \beta_j = 1}$$ and a scalar $\eps >0$. We say that the norm $\norm{\cdot}$ is
\begin{enumerate}
\item \emph{logarithmically optimal (log-optimal) for $\mathcal{P}$} if
$$\max_{A\in\mathcal{P}} \alpha(A) = \max_{j \in \{1,\dots,m\}} \lognorm{ }{A_j};$$
\item \emph{logarithmically $\eps$-optimal (log-$\eps$-optimal) for $\mathcal{P}$} if $$\displaystyle \max_{A\in\mathcal{P}} \alpha(A) \leq \max_{j \in \{1,\dots,m\}} \lognorm{ }{A_j} \leq \max_{A\in\mathcal{P}} \alpha(A) + \eps.$$
\end{enumerate}
\ed
We are specifically interested in the matrix polytopes
defined as ${\PF := \setdef{[d]W}{d\in[0,1]^n}}$ and ${\PH := \setdef{W[d]}{d\in[0,1]^n}}$.
Namely, in Theorem \ref{thm:prop_lognorm} we give algebraic results on the Euclidean log-norm of matrices in $\PF$ and $\PH$ (the proof is in Section~\ref{sec:additional_results}, together with a number of instrumental results).
\begin{rem}
It is always possible to rewrite $\PF$ and $\PH$ in the form of Definition~\ref{def:log-opt-polytopes}. In fact, let 
$A_1, \dots, A_{2^n} \in \R^{n \times n}$ be the $2^n$ vertices defined by $A_j = [v_j]W$ where $v_j \in \{0,1\}^n$ is the binary vector with entries either $0$ or $1$ (note that there are $2^n$ such binary vectors).
Then the set $\setdef{\sum_{j=1}^{2^n} \beta_j A_j}{\beta_j \geq 0, \sum_{j=1}^{2^n} \beta_j = 1}$ is exactly the set $\mathcal{P}_F := \setdef{[d]W}{d\in[0,1]^n}$.
To prove this, note that the vertices of the convex set $[0,1]^n$ are the $2^n$ vectors $v_j$.
Therefore, given $d\in[0,1]^n$ there exist $\beta_j \geq 0$, $j =1,\dots, 2^n$, with $\sum_{j=1}^{2^n} \beta_j = 1$ such that $[d]= \sum_{j=1}^{2^n} \beta_j [v_j]$.
Thus, 
\begin{align*}
\mathcal{P}_F &:= \setdef{[d]W}{d\in[0,1]^n} = \Bigsetdef{\sum_{j=1}^{2^n} \beta_j [v_j]W}{\beta_j \geq 0, \sum_{j=1}^{2^n} \beta_j = 1, v_j \in \{0,1\}^n}\\
&= \Bigsetdef{\sum_{j=1}^{2^n} \beta_j A_j}{\beta_j \geq 0, \sum_{j=1}^{2^n} \beta_j = 1}.
\end{align*}
The same reasoning holds for $\PH$.
\end{rem}
\bt[Euclidean log-norm of matrix polytopes]
\label{thm:prop_lognorm} Given a symmetric synaptic matrix $W$ (Assumption~\ref{ass:symmetric_synaptic_matrix}), the following statements holds:
\begin{enumerate}
\item \label{fact:nnell2:lognorm_alpha(W)>0}
if $\lmax>0$, then 
$\norm{\cdot}_{2,\ql{\lmax}}$, with $\ql{\lmax} \in \R^{n \times n}$ defined in~\eqref{eq:ql}, is log-optimal for $\PF$, i.e.,
\[
\displaystyle \max_{d\in[0,1]^n} \wlognorm{2}{\ql{\lmax}}{[d]W} = \max_{d\in[0,1]^n}\alpha([d]W) = \lmax.
\]

In addition, if $W$ is invertible, then
$\norm{\cdot}_{2,\qr{\lmax}}$, with $\qr{\lmax} \in \R^{n \times n}$ defined in~\eqref{eq:qr}, is log-optimal for $\PH$, i.e.,
\[
\displaystyle \max_{d\in[0,1]^n} \wlognorm{2}{\qr{\lmax}}{W[d]} = \max_{d\in[0,1]^n}\alpha(W[d]) = \lmax;
\]
\item \label{fact:nnell2:lognorm_alpha(W)=0}
if $\lmax = 0$, then for each $\eps>0$ the norm $\norm{\cdot}_{2,\ql{\eps}}$, with $\ql{\eps} \in \R^{n \times n}$ defined in~\eqref{eq:ql}, is log $\eps$-optimal for $\PF$, i.e.,
\[
\displaystyle \max_{d\in[0,1]^n} \wlognorm{2}{\ql{\eps}}{[d]W} \leq \max_{d\in[0,1]^n} \alpha([d]W) + \eps = \eps;
\]
\item \label{fact:nnell2:lognorm_alpha(W)<0}
if $\lmax < 0$, then $\norm{\cdot}_{2,(-W)^{1/2}}$ is log-optimal for $\PF$ and $\PH$, i.e.,
\begin{align*}
\displaystyle \max_{d\in[0,1]^n} \wlognorm{2}{(-W)^{1/2}}{[d]W} &= \max_{d\in[0,1]^n} \alpha([d]W) = 0,\\
\displaystyle \max_{d\in[0,1]^n} \wlognorm{2}{(-W)^{1/2}}{W[d]} &= \max_{d\in[0,1]^n} \alpha(W[d]) = 0.
\end{align*}
\end{enumerate}
\et
\begin{rem}
Theorem~\ref{thm:prop_lognorm} applies to polytopes of the form $aI_n{+}[d]W$ and of the form $aI_n{+}W[d]$, for all $a \in \R$. This follows from the log-norm translation property, i.e., for all $A \in \R^{n\times n}$ $\mu(A+a I_n) = \mu(A)+a$.
\end{rem}

\subsection{Contractivity of recurrent neural networks}
Next, we consider the neural network dynamics for the \FNN in  \eqref{eq:firing_rate_nn} and for the \HNN in \eqref{eq:hopfield_nn}.

\begin{assumption}[Slope-restricted activation function]
\label{ass:act_func_slope_restricted} 
The activation function $\map{\phi}{\R}{\R}$ is Lipschitz and slope restricted in $[0,1]$, i.e., 
$$\displaystyle 0 \leq \frac{\phi(x) - \phi(y)}{x-y} \leq 1, \text{ for all } x,y \in\R, x\neq y.$$
\end{assumption}
Assumption~\ref{ass:act_func_slope_restricted} ensures that $\phi'(x) \in[0,1]$ for almost all $x \in \R$.
Many common activation functions including ReLU, and sigmoid, satisfy Assumption~\ref{ass:act_func_slope_restricted}, possibly after rescaling.
In fact, Assumption~\ref{ass:act_func_slope_restricted} can be relaxed for larger classes of coupling by restricting the slope to $[0, \bar d]$, where $\bar d>0$. By defining $[d] := D\Phi/{\bar d}$ and $W := \bar d W$ our following results still hold for this general case, with $\alpha(W)$ replaced by $\alpha(\bar d \cdot W) = {\bar d} \cdot \alpha(W)$.
We assume $\bar d = 1$ to simplify the notation.
\subsubsection{Contractivity of firing rate neural networks}
We now provide an upper bound on the $\ell_2$ one-sided Lipschitz constant and sufficient conditions for the Euclidean contractivity of FNNs with symmetric weights.
\begin{thm}[Euclidean one-sided Lipschitz constant of the \FNN
]\label{thm:ell2_osLip_fr}
Consider the \FNN~\eqref{eq:firing_rate_nn} satisfying Assumptions~\ref{ass:symmetric_synaptic_matrix},~\ref{ass:act_func_slope_restricted}: 
\begin{enumerate}
\item \label{fact:nnell2:osLFR_alpha(W)>0}
if $\lmax >0$, then $$\displaystyle \osLip_{2,\ql{\lmax}}(\ffr) \leq -1{+}\lmax,$$ with $\ql{\lmax}\in\R^{n\times n}$ defined in~\eqref{eq:ql};
\item \label{fact:nnell2:osLFR_alpha(W)=0}
if $\lmax = 0$, then $$\displaystyle \osLip_{2,\ql\eps}(\ffr) \leq -1{+}\eps,$$ with ${\ql\eps\in\R^{n\times n}}$ defined in~\eqref{eq:ql};
\item \label{fact:nnell2:osLFR_alpha(W)<0}
if $\lmax < 0$, then $$\displaystyle \osLip_{2,(-W)^{1/2}}(\ffr) \leq -1.$$
\end{enumerate}
\end{thm}
\begin{proof}
Regarding part~\ref{fact:nnell2:osLFR_alpha(W)>0} note that for almost all $x \in \R^n$ we have
\begin{align*}
\wlognorm{2}{\ql{\lmax}}{D{\ffr}(x)} &= \wlognorm{2}{\ql{\lmax}}{-I_n+ D\Phi(Wx+u)W}\\
&\leq \max_{d\in[0,1]^n} \wlognorm{2}{\ql{\lmax}}{-I_n+[d]W}\\
&= -1{+}\lmax,
\end{align*}
where the last equality follows by the log-norm translation property and part~\ref{fact:nnell2:lognorm_alpha(W)>0} in Theorem~\ref{thm:prop_lognorm}. The proof follows by applying Theorem~\ref{th:equivalence_loc_lip_func}.
Parts~\ref{fact:nnell2:osLFR_alpha(W)=0} and~\ref{fact:nnell2:osLFR_alpha(W)<0} can be proved similarly, using parts~\ref{fact:nnell2:lognorm_alpha(W)=0} and~\ref{fact:nnell2:lognorm_alpha(W)<0} in Theorem~\ref{thm:prop_lognorm}.
\end{proof}
\begin{rem}
\label{rem:tight_ine_fnn}
Under further assumptions on the synaptic matrix and the activation function, some inequalities in Theorem~\ref{thm:ell2_osLip_fr} are tight -- see Appendix~\ref{apx:tight_ine_fnn}.
\end{rem}
The next result follows from Theorem \ref{thm:ell2_osLip_fr}.
\begin{cor}[Euclidean contractivity of the \FNN ]\label{cor:ell2_contractivity_fr}
Under the same assumptions and notations as in Theorem~\ref{thm:ell2_osLip_fr},
\begin{enumerate}
\item if $\lmax = 1$, then the \FNN is weakly infinitesimally contracting with respect to $\norm{\cdot}_{2,\ql{\lmax}}$;
\item \label{eq:contractivy_fr_0_1}
if ${0 < \lmax < 1}$, then the \FNN is strongly infinitesimally contracting with rate ${1-\lmax>0}$ with respect to ${\norm{\cdot}_{2,\ql{\lmax}}}$;
\item if $\lmax = 0$, then for any $0<\eps < 1$ the \FNN is strongly infinitesimally contracting with rate $1-\eps>0$ with respect to $\norm{\cdot}_{2,\ql\eps}$;
\item if $\lmax < 0$, then the \FNN is strongly infinitesimally contracting with rate $1$ with respect to ${\norm{\cdot}_{2,(-W)^{1/2}}}$.
\end{enumerate}
\end{cor}

\subsubsection{Contractivity of Hopfield neural networks}
We first provide an upper bound on the Euclidean one-sided Lipschitz constant and sufficient conditions for the $\ell_2$ contractivity of HNNs with non-singular symmetric synaptic matrix. Then, we give sufficient conditions for the $\ell_2$ contractivity with singular symmetric synapses. This latter result is proven in Section~\ref{sec:additional_results}: differently from our analysis on FNNs, it requires a distinct mathematical approach.
\begin{thm}[Euclidean one-sided Lipschitz constant of the \HNN with non-singular symmetric weights]\label{thm:ell2_osLip_alpha(W)_hop}
Consider the \HNN \eqref{eq:hopfield_nn} satisfying Assumptions~\ref{ass:symmetric_synaptic_matrix},~\ref{ass:act_func_slope_restricted} with non-singular weight matrix $W$,
\begin{enumerate}
\item \label{fact:nnell2:osLH_alpha(W)>0}
if $\lmax >0$, then $$\displaystyle \osLip_{2,\qr{\lmax}}(\fh) \leq -1{+}\lmax,$$ with $\qr{\lmax}\in\R^{n\times n}$ defined in~\eqref{eq:qr};
\item \label{fact:nnell2:osLH_alpha(W)<0}
if $\lmax < 0$, then
$$\displaystyle \osLip_{2,(-W)^{1/2}}(\fh) \leq -1.$$
\end{enumerate}
\end{thm}
\begin{proof} 
Regarding part~\ref{fact:nnell2:osLH_alpha(W)>0}, note that for almost all $x \in \R^n$ we have
\begin{align*}
\wlognorm{2}{\qr{\lmax}}{D{\fh}(x)} &= \wlognorm{2}{\qr{\lmax}}{-I_n+ W D\Phi(x)}\\
&\leq \max_{d\in[0,1]^n} \wlognorm{2}{\qr{\lmax}}{-I_n+ W[d]}\\
&= -1{+}\lmax,
\end{align*}
where the last equality follows by the log-norm translation property and part~\ref{fact:nnell2:lognorm_alpha(W)>0} in Theorem~\ref{thm:prop_lognorm}.
The proof then follows by applying Theorem~\ref{th:equivalence_loc_lip_func}.
Part~\ref{fact:nnell2:osLH_alpha(W)<0} can be proved similarly, using part~\ref{fact:nnell2:lognorm_alpha(W)<0} in Theorem~\ref{thm:prop_lognorm}.
\end{proof}
\begin{rem}
\label{rem:tight_ine_hnn}
Following the same reasoning as in Appendix~\ref{apx:tight_ine_fnn}, under the same assumptions of Theorem~\ref{thm:ell2_osLip_alpha(W)_hop}, if the activation function satisfies $\inf_{x \in \R} \phi'(x) = 0$, and $\sup_{x \in \R} \phi'(x) = 1$, then the inequalities in Theorem~\ref{thm:ell2_osLip_alpha(W)_hop} are tight.
\end{rem}
\begin{cor}[Euclidean contractivity of the \HNN with non-singular symmetric weights]\label{cor:ell2_contractivity_hop}
Under the same assumptions and notations as in Theorem~\ref{thm:ell2_osLip_alpha(W)_hop},
\begin{enumerate}
\item if $\lmax = 1$, then the \HNN is weakly infinitesimally contracting with respect to ${\norm{\cdot}_{2,\qr{\lmax}}}$;
\item if $0 < \lmax < 1$, then the \HNN is strongly infinitesimally contracting with rate $1-\lmax>0$ with respect to ${\norm{\cdot}_{2,\qr{\lmax}}}$;
\item if $\lmax < 0$, then the \HNN is strongly infinitesimally contracting with rate $1$ with respect to ${\norm{\cdot}_{2,(-W)^{1/2}}}$.
\end{enumerate}
\end{cor}

Finally, we give sufficient infinitesimal contractivity conditions of the \HNN with singular symmetric synapses (see Section~\ref{sec:additional_results} for the proof).
\bt[Contractivity of the \HNN with singular symmetric weights]
\label{thm:Hgen_ell2}
Consider the \HNN \eqref{eq:hopfield_nn} satisfying Assumptions~\ref{ass:symmetric_synaptic_matrix},~\ref{ass:act_func_slope_restricted} with $W$ having kernel $\mathcal{K} \neq \{\0_n\}$, and such that $\lmax<1$. Then, for each $\eps \in {]0, 1-\alpha(W)[}$ the \HNN is strongly infinitesimally contracting with rate $|1{-}\alpha(W){-}\eps|$.
\et
\begin{rem}
If $W = 0$, then the \FNN~\eqref{eq:firing_rate_nn} and the \HNN~\eqref{eq:hopfield_nn} are contracting with rate 1. As a consequence of Corollaries~\ref{cor:ell2_contractivity_fr},~\ref{cor:ell2_contractivity_hop} and Theorem~\ref{thm:Hgen_ell2}, when coupling is added to the networks, they remain (strongly) contracting as long as ${\alpha(W)} < 1$. Note that the entries of $W$ are allowed to be large, so as the activation function and this allows to have different types of coupling as long as the matrix $I_n - W$ is Hurwitz.
\end{rem}
\section{Proofs and Additional Results}\label{sec:additional_results}
We now present additional algebraic results on matrix polytopes and symmetric matrices, and the proofs of Theorems~\ref{thm:prop_lognorm} and~\ref{thm:Hgen_ell2}. First, we give a technical result for the spectral abscissa of matrix polytopes.
\begin{lem}[Lower bound on spectral abscissa of polytope of matrices]
\label{lemma:lower-bound-alpha-d} For any $W\in\R^{n\times{n}}$, we have
\begin{align}
&\max_{d\in[0,1]^n} \alpha([d]W) \geq \lmax_+ \label{eq:lower-bound-alpha-d_fr},\\
&\max_{d\in[0,1]^n} \alpha(W[d]) \geq \lmax_+.\label{eq:lower-bound-alpha-d_h}
\end{align}
\end{lem}
\begin{proof}
First, note that the spectral abscissa is a continuous function and that the set $\PF$
is compact, hence the maximum is well defined. To prove~\eqref{eq:lower-bound-alpha-d_fr} we compute:
\begin{align*}
\max_{d\in[0,1]^n} \alpha([d]W)& \geq \max \{\, \alpha([d]W) |_{d=\0_n}, \alpha([d]W) |_{d=\1_n}\,\} \\
&= \max \{\, 0, \lmax\,\} = \lmax_+.
\end{align*}
The same calculation applies to prove inequality~\eqref{eq:lower-bound-alpha-d_h}.
\end{proof}

We now give the proof of Theorem~\ref{thm:prop_lognorm}. To enhance clarity we prove its parts case by case.
Lemma~\ref{lemma:splitting} and parts~\ref{fact:nnell2:lognorm_alpha(W)>0} and~\ref{fact:nnell2:lognorm_alpha(W)=0} in Theorem~\ref{thm:prop_lognorm}, are based upon and extend the treatment in~\cite[Theorem~2]{LK-ME-JJES:22} -- see our statement of contributions.
\begin{lem}[Splitting upper-bounded symmetric matrices]
\label{lemma:splitting}
Consider $W$ satisfying Assumptions~\ref{ass:symmetric_synaptic_matrix}. Assume $W\preceq \lm I_n$, for some $\lm > 0$ and let $\fsplit_{\lm}(\cdot)$ and $\ql{\lm}$ be defined in~\eqref{eq:theta} and~\eqref{eq:ql}, respectively. Then,
\begin{equation}
\label{eq:function_decomposition_W}
W = \ql{\lm}{-}\frac{1}{4\lm} \ql{\lm}^2.
\end{equation}
\end{lem}
\begin{proof}
By definition of the function $\fsplit_{\lm}(\cdot)$, for all $\lambda_i \leq \lm$, $i\in \until{n}$, it holds
\beq
\label{eq:1_proof_decompo_W}
\lambda_i = \fsplit_{\lm}(\lambda_i){-}\frac{1}{4{\lm}}\fsplit_{\lm}(\lambda_i)^2.
\eeq
In fact, we have
\begin{align*}
\fsplit_{\lm}(\lambda_i){-}\frac{1}{4{\lm}}\fsplit_{\lm}(\lambda_i)^2 & =  2\lm\Bigg(1+\sqrt{1{-}\frac{\lambda_i}{\lm}}\Bigg){-}\frac{1}{4{\lm}}4\lm^2\Bigg(1+\sqrt{1{-}\frac{\lambda_i}{\lm}}\Bigg)^2 \\
&= 2\lm\Bigg(1+\sqrt{1{-}\frac{\lambda_i}{\lm}}\Bigg) {-}\lm\Bigg(1+2\sqrt{1{-}\frac{\lambda_i}{\lm}} + 1{-}\frac{\lambda_i}{\lm} \Bigg)\\
&=\lm \Bigg(2+2\sqrt{1{-}\frac{\lambda_i}{\lm}}{-}2{+}\frac{\lambda_i}{\lm}{-}2\sqrt{1{-}\frac{\lambda_i}{\lm}} \Bigg)\\
&= \lambda_i.
\end{align*}
Equation~\eqref{eq:1_proof_decompo_W} implies $\Lambda = \fsplit_{\lm}(\Lambda){-}\frac{1}{4{\lm}}\fsplit_{\lm}(\Lambda)^2$. Equality~\eqref{eq:function_decomposition_W} follows by multiplying by $U$ and $U^\top$ to the left and to the right, respectively, with $U$ defined in~\eqref{eq:decomposition_W}.
\end{proof}
First, we prove part~\ref{fact:nnell2:lognorm_alpha(W)>0}, i.e., the log-optimality of the norm $\norm{\cdot}_{2,\ql{\lmax}}$ and, when $W$ is invertible, of $\norm{\cdot}_{2,\qr{\lmax}}$ for multiplicatively-scaled matrices with positive maximum eigenvalue.
\begin{proof}[Proof of part~\ref{fact:nnell2:lognorm_alpha(W)>0}]
First, we prove that
$\norm{\cdot}_{2,\ql{\lmax}}$ is log-optimal for $\PF$ and $\displaystyle \max_{d\in[0,1]^n}\alpha([d]W) = \lmax$.
To this purpose, define 
\[
P := \frac{1}{4\lmax}\ql{\lmax}^2 \succ 0.
\]
Lemma~\ref{lemma:splitting} implies $W=\ql{\lmax}-P$. Next, pick $d\in\R^n$ satisfying $\0_n<d\leq\1_n$, so that $[d]$ is diagonal and invertible. Then
\begin{align}
&2\lmax P{-}\frac{1}{2}\ql{\lmax}^2 \succeq 0  \label{eq:starting-LMI_alpha(W)>0} \\
&\!\!\!\!\implies 2\lmax P{-}\frac{1}{2}\ql{\lmax}[d]\ql{\lmax}  \succeq 0   \nonumber \\
&\!\!\!\!\iff 2\lmax P{-}\ql{\lmax}[d]P(2P[d]P)^{-1}P[d]\ql{\lmax}  \succeq 0. \nonumber
\end{align}
Since $P[d]P \succ 0$, we can apply the Schur complement to this LMI to conclude that
\begin{equation} \label{eq:Kaz-trick_alpha(W)>0}
y^\top
\begin{bmatrix}
2\lmax P & -\ql{\lmax}[d]P \\
-P[d]\ql{\lmax} & 2P[d]P
\end{bmatrix}
y \geq 0, \quad \forall y \in \R^{2n}.
\end{equation}
Setting $y = (y_1,y_1)$ for arbitrary $y_1 \in \R^n$, the inequality~\eqref{eq:Kaz-trick_alpha(W)>0} implies
\begin{align}
&2\lmax P{-}\ql{\lmax}[d]P{-}P[d]\ql{\lmax}{+}2P[d]P \succeq 0 \nonumber \\
&\!\!\!\!\iff \ql{\lmax}[d]P{+}P[d]\ql{\lmax}{-}2P[d]P \preceq 2\lmax P \nonumber \\
&\!\!\!\!\overset{W = \ql{\lmax}-P}{\iff} W[d]P{+}P[d]W \preceq 2 \lmax P \nonumber \\
&\!\!\!\!\iff \ql{\lmax}^2[d]W{+}W[d]\ql{\lmax}^2 \preceq 2 \lmax\ql{\lmax}^2.
\label{eq:final-LMI_alpha(W)>0}
\end{align}
In summary, we have established that the weak LMI~\eqref{eq:starting-LMI_alpha(W)>0} (independent of $d$) implies the weak LMI~\eqref{eq:final-LMI_alpha(W)>0} for all $0<d\leq\1_n$. Here, by weak LMI, we mean to state that the linear matrix inequality is not strict. It is known~\cite[Theorem~6.3.5]{RAH-CRJ:12} that the eigenvalues of a symmetric matrix are continuous functions of the matrix entries. Therefore, the LMI~\eqref{eq:final-LMI_alpha(W)>0} holds also for $\0_n\leq d\leq\1_n$. Finally, note that the LMI~\eqref{eq:final-LMI_alpha(W)>0} is equivalent to the condition ${\wlognorm{2}{\ql{\lmax}}{[d]W} \leq \lmax}$ for all ${d\in[0,1]^n}$, therefore
\[
\max_{d \in [0,1]^n} \wlognorm{2}{\ql{\lmax}}{[d]W} \leq \lmax.
\]
Moreover, it is well known~\cite{CAD-MV:1975} that for every log-norm $\mu$ and every matrix $A$ it holds $\alpha(A) \leq \mu(A)$. Specifically in our case:
\[
\max_{d \in [0,1]^n}\alpha([d]W) \leq \max_{d \in [0,1]^n} \wlognorm{2}{\ql{\lmax}}{[d]W}.
\]
The proof then follows from~\eqref{eq:lower-bound-alpha-d_fr}, after noticing that in this case $\lmax_{+} = \lmax$.

Next, assume that $W$ is invertible. We need to prove that $\norm{\cdot}_{2,\qr{\lmax}}$ is log-optimal for $\PH$ and that it holds $\displaystyle \max_{d\in[0,1]^n}\alpha(W[d]) = \lmax$.
We have
\begin{align*}
\max_{d\in[0,1]^n} \wlognorm{2}{\qr{\lmax}}{W[d]} & = \max_{d\in[0,1]^n} 
\wlognorm{2}{\ql{\lmax}W^{-1}}{W[d]}\\
&\overset{\eqref{prop:silly-weightedmu}}{=} \max_{d\in[0,1]^n}\wlognorm{2}{\ql{\lmax}}{[d]W}\\ &= \lmax,
\end{align*}
where the last equality follows from the log-optimality of $\norm{\cdot}_{2,\ql{\lmax}}$ for $\PF$. The proof again follows from~\eqref{eq:lower-bound-alpha-d_fr}.
\end{proof}

The proof of part~\ref{fact:nnell2:lognorm_alpha(W)=0} of Theorem~\ref{thm:prop_lognorm}, i.e., the log-optimality of the weighted $\ell_2$ norm $\norm{\cdot}_{2,\ql{\eps}}$ for multiplicatively-scaled negative semidefinite matrices, follows the same reasoning as that of part~\ref{fact:nnell2:lognorm_alpha(W)>0} by considering $\eps>0$ instead of $\lmax$. Hence, we omit it here for brevity.

Finally, we prove part~\ref{fact:nnell2:lognorm_alpha(W)<0}, i.e., the log-optimality of $\norm{\cdot}_{2,(-W)^{1/2}}$ for multiplicatively-scaled negative definite matrices. 
To do so, we give the following algebraic result.
\begin{lem}[Optimal norms for products of symmetric matrices]
\label{lem:prod_sym_matrices}
Let $A_1 = S\Q \in \R^{n\times n}$ and $A_2 = \Q S \in \R^{n\times n}$ where $S$, $\Q\in\S^n$, with $Q \succ 0$. Then, for each $i\in \{\,1,2\,\}$,
\begin{enumerate}
\item \label{stat1-lem:prod_sym_matrices}
$\spec(A_i)$ is real and has the same number of negative, zero, and positive eigenvalues as $S$;
\item \label{stat2-lem:prod_sym_matrices}
the norm $\norm{\cdot}_{2,\Q^{1/2}}$ is optimal for the matrix $A_i$, i.e., $\norm{A_i}_{2,\Q^{1/2}} = \rho(A_i)$;
\item \label{stat3-lem:prod_sym_matrices}
the norm $\norm{\cdot}_{2,\Q^{1/2}}$ is log-optimal for $A_i$, i.e., $\wlognorm{2}{\Q^{1/2}}{A_i} = \alpha(A_i)$.
\end{enumerate}
\end{lem}
\begin{proof}
Let $i = 1$. $A_1$ is similar to $\Q^{1/2}S \Q^{1/2} \in \S^n$, hence $\spec(A_1)$ is real. Part~\ref{stat1-lem:prod_sym_matrices} then follows from Sylvester’s law of inertia, noting that $\Q^{1/2}S\Q^{1/2}$ is congruent to $S$.
Regarding part~\ref{stat2-lem:prod_sym_matrices}, we compute
\begin{align*}
\norm{A_1}_{2,\Q^{1/2}}^2 &= \subscr{\lambda}{max}(\Q^{-1}A_1^\top \Q A_1) = \subscr{\lambda}{max}(\Q^{-1}(\Q S)\Q (S\Q))\\
&=\subscr{\lambda}{max}((S\Q)^2) = \rho(S\Q)^2,
\end{align*}
where the last equality follows from the fact that $(S\Q)^2$ has the same eigenvectors as $S\Q$ and real eigenvalues equal to the square of the real eigenvalues of $S\Q$. Finally, to prove part~\ref{stat3-lem:prod_sym_matrices} we compute
\begin{align*}
\wlognorm{2}{\Q^{1/2}}{A_1} &= \subscr{\lambda}{max}\Big(\frac{\Q A_1 \Q^{-1}{+}A_1^\top}{2}\Big) = \subscr{\lambda}{max}\Big(\frac{\Q (S \Q)\Q^{-1}{+}\Q S}{2}\Big)\\
&= \subscr{\lambda}{max}(\Q S) =  \subscr{\lambda}{max}(\Q S \Q\Q^{-1})\\
& = \subscr{\lambda}{max}(\Q A_1\Q^{-1}) = \subscr{\lambda}{max}(A_1) \\
&= \alpha(A_1).
\end{align*}
This concludes the proof of part~\ref{stat2-lem:prod_sym_matrices}.
The proof for $i = 2$ is a straightforward adaptation.
\end{proof}
\begin{proof}[Proof of part~\ref{fact:nnell2:lognorm_alpha(W)<0}]
Pick $d\in\R^n$ satisfying $\0_n\leq d\leq\1_n$ and consider the matrices $[d]W$ and $W[d]$. Lemma~\ref{lem:prod_sym_matrices} with ${S:=[-d]}$ and ${\Q:=-W \succ 0}$, implies that the spectrum of the product matrices ${[d]W =[-d](-W)}$ and ${W[d] =(-W)[-d]}$ is real and has the same number of negative, zero, positive eigenvalues as ${[-d]}$. Therefore,
\begin{align}
\wlognorm{2}{(-W)^{1/2}}{[d]W} &= \alpha([d]W)
\left\{\,
\begin{matrix}
< 0    & \text{ if } d>\0_n,\\
\leq 0 & \text{ otherwise,} 
\end{matrix}
\right. \\
\wlognorm{2}{(-W)^{1/2}}{W[d]} &= \alpha(W[d])
\left\{\,
\begin{matrix}
< 0    & \text{ if } d>\0_n,\\
\leq 0 & \text{ otherwise.} 
\end{matrix}
\right.
\end{align}
Maximizing over $d \in [0,1]^n$ we get part~\ref{fact:nnell2:lognorm_alpha(W)<0}.
\end{proof}
Finally, we give the proof of Theorem~\ref{thm:Hgen_ell2}.
\begin{proof}[Proof of Theorem~\ref{thm:Hgen_ell2}]
Let $r$ be the number of non-zero eigenvalues of $W \in \R^{n\times n}$.
Without loss of generality, we reorder the elements in $\lambda \in \R^n$ and $U\in \R^{n\times n}$, so that $\lambda = (\lambda_1, \dots, \lambda_r, 0,\dots, 0)$ and $U = [u_1, \dots, u_r, u_{r+1},\dots,u_n]$, where $u_i \in \R^n$ is the eigenvector of $\lambda_i \in \R$.

Next, let $\K^{*} := \spn\{\,u_1,\dots,u_r\,\}$, $\nparal := \dim(\K^*)$, ${\K := \spn \{\,u_{r+1},\dots,u_n\,\}}$, $\nperp := \dim(\K)$, and define $\Uparal := [u_1, \dots, u_r] \in \R^{n\times \nparal}$, $\Uperp := [u_{r+1},\dots,u_n]\in \R^{n\times \nperp}$, so that $U = [\Uparal \quad \Uperp]$.

We have $\R^{n} = \{\,x \in \R^n \ | \ x \in \K^*\,\} \oplus \{\,x \in \R^n\ |\ x \in \K\,\}$. Therefore, given $x \in \R^n$ we can always define ${\xparal = \Uparal^{\top}x \in \K^{*}}$ and $\xperp = \Uperp^{\top}x \in \K$.
We note that $\trasp{U} U = I_n$ implies  $\trasp{\Uparal}\Uparal= I_{\nparal}$, $\trasp{\Uperp} \Uperp = I_{\nperp}$, $\trasp{\Uperp} \Uparal = \0_{\nperp \times \nparal}$, and $\trasp{\Uparal}\Uperp = \0_{\nparal \times \nperp}$. Also,
\begin{align*}
W &= [\Uparal \quad \Uperp]
\begin{bmatrix}
\Lambda_{\parallel} & 0_{\nparal \times \nperp} \\
0_{\nperp \times \nparal} & 0_{\nperp \times \nperp}
\end{bmatrix}
\begin{bmatrix}
\Uparal^{\top} \\
\Uperp^{\top}
\end{bmatrix}
= \Uparal \Lambda_{\parallel} \trasp{\Uparal},\\
\ql{\alpha(W)} &= U \theta_{\alpha(W)}(\Lambda) U^\top = [\Uparal \quad \Uperp]
\begin{bmatrix}
\theta_{\parallel} & 0_{\nparal \times \nperp} \\
0_{\nperp \times \nparal} & \theta_{\perp}
\end{bmatrix}
\begin{bmatrix}
\Uparal^{\top} \\
\Uperp^{\top}
\end{bmatrix}\\
&= \Uparal \theta_{\parallel} \Uparal^\top{+} \Uperp \theta_\perp \Uperp^\top.
\end{align*}

Moreover, we have
\beq \label{eq:log_max_fr_theta_paral}
\max_{d\in[0,1]^n}
\wlognorm{2}{\theta_\parallel}{- I_{\nparal}{+}\Uparal^\top[d]\Uparal\Lambda_{\parallel}} \leq{-}1{+}\alpha(W).
\eeq
In fact, from Corollary~\ref{cor:ell2_contractivity_fr} we know:
\begin{align*}
&2\alpha(W)\ql{\alpha(W)}{+}\ql{\alpha(W)}[d]W{+}W[d]\ql{\alpha(W)} \preceq 0\\
&\iff 2\alpha(W) (\Uparal \theta^2_{\parallel} \Uparal^\top{+} \Uperp \theta^2_\perp \Uperp^\top)\\
&\quad \quad \quad+(\Uparal \theta^2_{\parallel} \Uparal^\top{+} \Uperp \theta^2_\perp \Uperp^\top)[d]\Uparal \Lambda_{\parallel} \trasp{\Uparal}\\
&\quad \quad \quad+\Uparal \Lambda_{\parallel} \trasp{\Uparal}[d](\Uparal \theta^2_{\parallel} \Uparal^\top{+}\Uperp \theta^2_\perp \Uperp^\top)\preceq 0.
\end{align*}
By multiplying by $\Uparal^\top$ and $\Uparal$ to the left and to the right, respectively, we get
\begin{align}
2\alpha(W)\theta^2_{\parallel}{+} \theta^2_{\parallel} \Uparal^\top [d]\Uparal \Lambda_{\parallel}{+}\Lambda_{\parallel} \trasp{\Uparal}[d]\Uparal \theta^2_{\parallel} \preceq 0 .
\end{align}
Thus,
$\wlognorm{2}{\theta_\parallel}{- I_{\nparal}{+}\Uparal^\top[d]\Uparal\Lambda_{\parallel}} \leq -1+\alpha(W)$.
Next, by multiplying~\eqref{eq:hopfield_nn} by $\Uperp^{\top}$ and $\Uparal^{\top}$ we obtain the interconnected system:
\[
\left\{\,
\begin{array}{cc}
\Uperp^{\top} \hopdot ={-}\Uperp^{\top} \hop{+}\Uperp^{\top} W\Phi(\hop){+}\Uperp^{\top}\uh, \\
\Uparal^{\top} \hopdot ={-}\Uparal^{\top} \hop{+}\Uparal^{\top} W\Phi(W\hop){+}\Uparal^{\top}\uh,
\end{array}
\right.
\]
thus,
\begin{numcases}{}
\hopdot^\perp ={-}\hop^\perp{+}\uh^\perp := \fh^\perp(\hop^\perp, \uh^\perp), \label{eq:fperp_hnn}\\
\hopdot^\parallel ={-}\hop^\parallel{+}\Lambda_{\parallel} \Uparal^{\top}\Phi(\hop){+}\uh^{\parallel} := \fh^\parallel(\hop, \uh^{\parallel}).
\label{eq:fparallel_hnn}
\end{numcases}
Equation~\eqref{eq:fperp_hnn} is always contracting with respect to any norm in the subspace $\K$ with $\osLip(\fh^\perp) ={-}1$, being 
$\mu(D \fh^\perp) = \mu(- I_{\nperp}) = -1$.
For system~\eqref{eq:fparallel_hnn} we define $\qr{\alpha(W)} := \ql{\alpha(W)}\mpi{W} = U  \theta_{\alpha{(W)}}\Lambda^{\dagger} \trasp{U}$, where $\mpi{W} = U \Lambda^{\dagger} \trasp{U}$, with
$$
\Lambda^{\dagger} =
\begin{bmatrix}
\Lambda_{\parallel}^{-1} & 0_{\nparal \times \nperp} \\
0_{\nperp \times \nparal} & 0_{\nperp \times \nperp}
\end{bmatrix}.
$$
Next, we note that the matrix $\subscr{Q}{H$\parallel$} := \Uparal^{\top} \ql{\alpha(W)} \mpi{W} \Uparal = \theta_\parallel \Lambda_{\parallel}^{-1}$ and that $D \fh^\parallel ={-}I_{\nparal}{+}\Lambda_{\parallel}\Uparal^\top[d]\Uparal$. Thus, we have
\begin{align*}
\osLip_{2,\subscr{Q}{H$\parallel$}}(\fh^\parallel) &\leq \max_{d\in[0,1]^n} \wlognorm{2}{\subscr{Q}{H$\parallel$}}{D \fh^\parallel}\\
&\leq \max_{d\in[0,1]^n}
\wlognorm{2}{\theta_\parallel \Lambda_{\parallel}^{-1}}{- I_{\nparal}{+}\Lambda_{\parallel}\Uparal^\top[d]\Uparal}\\
&\overset{\eqref{prop:silly-weightedmu}}{=} \max_{d\in[0,1]^n} 
\wlognorm{2}{\theta_\parallel}{- I_{\nparal}{+}\Uparal^\top[d]\Uparal\Lambda_{\parallel}}\\
&\overset{\eqref{eq:log_max_fr_theta_paral}}{\leq}{-}1{+}\alpha(W).
\end{align*}
Thus system~\eqref{eq:fparallel_hnn} is strongly infinitesimally contracting in $\K^{*}$ with respect to $\norm{\cdot}_{\subscr{Q}{H$\parallel$}}$ with rate $1-\alpha(W)$.

Finally, we note that at fixed $\xparal$ and $t$, the map $\xperp \to f_\parallel$ is Lipschitz with constant ${\textup{L}_{\parallel \perp} := \lmax}$. 
In fact, ${\forall \xperp^1, \xperp^2 \in \K}$, we get
\begin{align*}
\norm{f_{\parallel}(\xparal,\xperp^1){-}f_{\parallel}(\xparal,\xperp^2)} &= \norm{- \xparal{+}W\Phi(\xperp^1+\xparal){+}u{+}\xparal{-}W\Phi(\xperp^2+\xparal){-}u} \\
&= \norm{W(\Phi(\xperp^1+\xparal){-}\Phi(\xperp^2+\xparal))}\\
&\leq \lmax\norm{\Phi(\xperp^1+\xparal){-} \Phi(\xperp^2+\xparal)}\\
&\leq \lmax\norm{\xperp^1{-}\xperp^2}.
\end{align*}
We can now construct the gain matrix~\eqref{eq:def_gain_matrix}
\beq
\label{eq:gain_matrix}
\Gamma =
\begin{bmatrix}
- 1 & 0\\
\lmax &{-}1+\alpha(W)
\end{bmatrix}
\in \R^{2 \times 2}.
\eeq
The eigenvalues of $\Gamma$ are $\lambda_1 = -1, \lambda_2 ={-}1{+}\alpha(W)$.
The fact that $\K\neq \{\0_n\}$ implies $\alpha(W) \geq 0$. In turn, since by assumptions $\lmax<1$, we have $\lambda_2 \in {[-1, 0[}$. Thus $\Gamma$ is Hurwitz and $\alpha(\Gamma) ={-}1 +\alpha(W)$.
By applying Theorem~\ref{th:contrac_interc_system}, for each $\eps \in {]0, 1-\alpha(W)[}$ we have that the \HNN is strongly infinitesimally contracting with rate $|\alpha(\Gamma)+\eps|$. This concludes the proof.
\end{proof}
\section{Using Euclidean contractivity to solve quadratic optimization problems}
We now apply the previous results to propose a firing-rate neural network solving certain quadratic optimization problems with box constraints.
By utilizing Corollary~\ref{cor:ell2_contractivity_fr}, we ensure global exponential convergence of our dynamic, along with all the other properties of contracting systems.

Given $A = A^\top \succ 0$, an input $u \in \R^n$, and $\mu \leq \nu \in \R^n$ the \emph{quadratic optimization problem with box constraints} is
\beq\label{eq:quadratic_op_with_linear_constraints}
\min_{y \in \R^n} \Big(J_{A,u}(y) := \frac{1}{2} y^\top A y - u^\top y \Big), \quad \text{s.t. } \mu \leq y \leq \nu.
\eeq

Note that $J_{A,u}(\cdot)$ is strongly convex and the constraints are convex, thus~\eqref{eq:quadratic_op_with_linear_constraints} admits a unique global optimal solution.

We propose the following \FNN model to solve~\eqref{eq:quadratic_op_with_linear_constraints}.
Given a single-layered neural network of $n$ neurons, the state ${x \in \R^n}$ evolves according to
\beq
\label{eq:fr_constrained_qop_x_dot}
\dot{x} = -x +\sat{\mu, \nu}{(I_n-A)x{+}u},
\eeq
with output $y = x$. The activation function $\map{\sat{\mu, \nu}{\cdot}}{\R^n}{[\mu, \nu]:=[\mu_1,\nu_1]\times \dots \times [\mu_n,\nu_n]}$, illustrated in Figure~\ref{fig:plot_sat}, is defined as $(\sat{\mu, \nu}{x})_i= \sat{\mu_i, \nu_i}{x_i}$, where $\map{\sat{\mu_i, \nu_i}{\cdot}}{\R}{[\mu_i, \nu_i]}$ is
\[
\sat{\mu_i, \nu_i}{x_i} =
\left\{
\begin{array}{ll}
\mu_i & \textup{ if } x_i \leq \mu_i,\\
x_i     & \textup{ if } \mu_i < x_i < \nu_i,\\
\nu_i & \textup{ if } x_i \geq \nu_i.
\end{array}
\right.
\]
To simplify the notation, whenever it is clear from the context, we use the same symbol for both the scalar and vector forms of the saturation function.
\begin{figure}[!ht]
\scalebox{0.55}{\input{plot_sat_mu_nu.pgf}}
\centering
\caption{Saturation function $\sat{\mu, \nu}{\cdot}$ with $\mu = -1$ and $\nu = 3$.}
\label{fig:plot_sat}
\end{figure}
\begin{rem}
The function $\sat{\mu_i, \nu_i}{\cdot}$ satisfies Assumption~\eqref{ass:act_func_slope_restricted}.
Almost everywhere, its partial derivative is ${\partial \sat{a, b}{\cdot} \colon \R \setminus \{\, a, b \,\} \to \{0,1\}}$ defined by
\beq
\label{eq:der_sat_function}
\derp{(\sat{a, b}{z})}{z} =
\left\{
\begin{array}{cc}
0 & \textup{ if } z \notin {]a,b[}, \\
1 & \textup{ if } z \in {]a,b[}.
\end{array} 
\right.
\eeq
\end{rem}
Next, we use Corollary~\ref{cor:ell2_contractivity_fr} to give sufficient conditions for the strong infinitesimal contractivity of~\eqref{eq:fr_constrained_qop_x_dot}. Then, we show that the equilibrium of~\eqref{eq:fr_constrained_qop_x_dot} is the optimal solution of~\eqref{eq:quadratic_op_with_linear_constraints}.
\begin{lem}[Strong infinitesimal contractivity]
\label{th:contractivity_fr_constrained_qop}
Let ${A = A^\top \succ 0}$ in~\eqref{eq:fr_constrained_qop_x_dot}. The FNN~\eqref{eq:fr_constrained_qop_x_dot} is strongly infinitesimally contracting with rate $c > 0$ with respect to thee norm $\norm{\cdot}_{2,P}$, where
\begin{enumerate}[wide]
\item if $\subscr{\lambda}{min}(A) < 1$, then $c= \subscr{\lambda}{min}(A)$ and ${P = \ql{1{-}\subscr{\lambda}{min}(A)}}$, with $\ql{1{-}\subscr{\lambda}{min}(A)}$ defined in~\eqref{eq:ql};
\item if $\subscr{\lambda}{min}(A) = 1$, then for any $0<\eps < 1$, $c= 1-\eps>0$ and $P = \ql{\eps}$, with $\ql{\eps}$  defined in~\eqref{eq:ql};
\item if $\subscr{\lambda}{min}(A) > 1$, then $c=1$ and $P = (A - I_n)^{1/2}$.
\end{enumerate}
\end{lem}
\begin{proof}
The thesis follows by applying Corollary~\ref{cor:ell2_contractivity_fr} noticing that $A \succ 0$ implies $W = I_n{-}A \prec I_n$, thus $\alpha(W)=1{-}\subscr{\lambda}{min}(A) < 1$, and $\sat{\mu,\nu}{\cdot}$ satisfies Assumption~\ref{ass:act_func_slope_restricted}.
\end{proof}
An immediate consequence of Lemma~\ref{th:contractivity_fr_constrained_qop} is that~\eqref{eq:fr_constrained_qop_x_dot} admits a unique equilibrium point.
Next, we prove that this equilibrium point is  the optimal solution of~\eqref{eq:quadratic_op_with_linear_constraints}.
\begin{lem}
\label{lem:equivalence_fr_constrained_qop}
The vector $\xstar \in\R^n$ is the global minimum for~\eqref{eq:quadratic_op_with_linear_constraints} if and only if $\xstar$ is the equilibrium point of~\eqref{eq:fr_constrained_qop_x_dot}.
\end{lem}
\begin{proof}
Let $\xstar \in \R^n$ be a global minimum for~\eqref{eq:quadratic_op_with_linear_constraints}, thus $\xstar \in [\mu, \nu]$. Then it follows from the KKT conditions that, for all $i \in \until{n}$,
\beq\label{eq:condition_min_pos_lasso}
\derp{J_{A,u}}{x_i}(\xstar) = (A\xstar)_i{-}u_i
\left\{
\begin{array}{ll}
\geq 0 & \text{ if } \xstar_i = \mu_i,\\
= 0    & \text{ if } \mu_i < \xstar_i < \nu_i, \\
\leq 0 & \text{ if } \xstar_i = \nu_i.
\end{array}
\right.
\eeq
Note that $\xstar$ is an equilibrium of~\eqref{eq:fr_constrained_qop_x_dot} if, for all $i$
, we have
\begin{align}
\label{eq:fr_constrained_qop_eq}
-\xstar_i +\sat{\mu_i, \nu_i}{\xstar_i{-}(A\xstar)_i{+}u_i} = 0.
\end{align}
If $\xstar_i = \mu_i$, let $z^{\star} : = \left.(A\xstar)_i\right|_{\xstar_i = \mu_i}{-}u_i$. By definition of $\sat{\mu_i, \nu_i}{\cdot}$ it holds ${-}\mu_i +\sat{\mu_i, \nu_i}{\mu_i{-}z^{\star}} \geq0$. Moreover, from the KKT conditions~\eqref{eq:condition_min_pos_lasso}, and being  $\sat{\mu_i, \nu_i}{\cdot}$ monotonically non-decreasing, we get the reverse inequality.
Thus $\xstar_i = \mu_i$ verifies~\eqref{eq:fr_constrained_qop_eq}.
Similarly it can be proved that~\eqref{eq:fr_constrained_qop_eq} holds for $\mu_i < \xstar_i < \nu_i$, and $\xstar_i = \nu_i$.

Vice versa, let $\xstar \in \R^n$ be an equilibrium of~\eqref{eq:quadratic_op_with_linear_constraints}, i.e.,~\eqref{eq:fr_constrained_qop_eq} holds. If $\xstar_i \leq \mu_i$, then~\eqref{eq:fr_constrained_qop_eq} implies
$
\xstar_i  =\sat{\mu_i, \nu_i}{\mu_i{-}z^{\star}}.
$
By definition of $\sat{\mu_i, \nu_i}{\cdot}$ we get $\xstar_i \in [\mu_i,\nu_i]$, thus $\xstar_i = \mu_i$, and $\mu_i{-}z^{\star} \leq \mu_i$, which implies $z^{\star} \geq 0$. 
Similarly, if $\mu_i < \xstar_i < \nu_i$, then 
$z^{\star} = 0$, 
while if $\xstar_i \geq \nu_i$, then $\xstar_i = \nu_i$ and $z^{\star} \leq 0$.
This ends the proof since we have shown that the KKT conditions~\eqref{eq:condition_min_pos_lasso} hold for all $i$.
\end{proof}
\section{Conclusion}
We presented sharp conditions for strong and weak Euclidean contractivity of Hopfield and firing-rate neural networks with symmetric weights together with a number of general algebraic results. Specifically, we analyzed the Euclidean log-norm of matrix polytopes, proposing norms that are log-optimal for almost all matrices, and provided optimal and log-optimal norms for the product of symmetric matrices.
We considered networks with (possibly) non-smooth activation functions, which allows us to consider common activation functions such as ReLU and the soft thresholding function. Finally, to demonstrate the practical implications of our results, we proposed a \FNN to solve quadratic optimization problems with box constraints.

As future work, it would be useful to (i) extend our results to arbitrary
synaptic matrices (as opposed to only symmetric) and heterogeneous
dissipation matrices, (ii) establish higher-order contractivity
properties~\cite{CW-IK-MM:22} and consider stochastic models~\cite{ZA:22},
and (iii) apply these results to neuroscience and machine learning
problems. For example, we plan to study sparse reconstruction networks
(inspired by~\cite{CJR-DHJ-RGB-BAO:08}) and implicit learning models (e.g.,
see~\cite{MR-IM:20}).
\appendices
\section{Interconnected systems}
In this section, we briefly review the theory of contracting interconnected systems, that we used to prove Theorem~\ref{thm:Hgen_ell2}.
We refer to~\cite{FB:23-CTDS} for a recent and more detailed review.

Given $r$ positive integers $n_1,\dots, n_r$ such that $n_1 + \dots + n_r = n$, consider the decomposition $\R^n = \R^{n_1} \times \dots \times \R^{n_r}$, a local norm $\norm{\cdot}_i$ on $\R^{n_i}$, for each $i \in \{1,\dots, r\}$, with associated log-norm $\mu_i(\cdot)$. Consider the \emph{interconnection of $r$ dynamical systems}
\beq
\label{eq:interconnected_system}
\dot x_i = f_i(t,x_i,x_{-i}), \quad \forall i \in \{1,\dots,r\},
\eeq
where $x_i \in \R^{n_i}$, and $x_{-i} \in \R^{n-n_i}$ denote the vector $x$ without the component $x_{i}$.
We recall the following results that will be useful for our analysis.
\bt[Contractivity of interconnected system]
\label{th:contrac_interc_system}
Consider the interconnected system in~\eqref{eq:interconnected_system}. Assume \begin{enumerate}[label=\textup{($A$\arabic*)}, leftmargin=1.4 cm,noitemsep]
\item\label{eq:ass_A_1} (contractivity-at-each-node) at fixed $x_{-i}$ and $t$, each function $x_i \to f_i(t, x_i, x_{-i})$ is strongly infinitesimally contracting with rate $c_i$ with respect to $\norm{\cdot}_i$.
\item\label{eq:ass_A_2} (Lipschitz interconnections) at fixed $x_i$ and $t$, each function $x_{-i} \to f_i(t, x_i, x_{-i})$ is Lipschitz with Lipschitz constant $\gamma_{ij} \in \R_{\geq 0}$.
\end{enumerate}
Define the \emph{gain matrix}
\beq
\label{eq:def_gain_matrix}
\Gamma =
\begin{bmatrix}
- c_1 & \dots & \gamma_{1r}\\
\vdots & \dots & \vdots \\
\gamma_{r1} & \dots & - c_r 
\end{bmatrix}
\in \R^{r \times r}.
\eeq
If $\Gamma$ is Hurwitz, then the interconnected system is strongly infinitesimally contracting with respect to $\norm{\cdot}_{\eta}$ and with rate $|\alpha(\Gamma) + \eps|$, where $\eta \in \R^n_{>0}$, $\norm{\cdot}_{\eta}^2 := \sum_{i=1}^r \eta_i\norm{x_i}_i^2$, and $\epsilon >0$.
\et

\newcommand{\nondiff}{\Omega_\phi}
\newcommand{\nondiffBig}{\Omega_\Phi}
\section{Justification for Remark~\ref{rem:tight_ine_fnn}}
\label{apx:tight_ine_fnn}
\begin{lem}
Given the \FNN~\eqref{eq:firing_rate_nn} with symmetric (Assumption~\ref{ass:symmetric_synaptic_matrix}) and invertible synaptic matrix $W$, Lipschitz and slope restricted in $[0,1]$ (Assumption~\ref{ass:act_func_slope_restricted}) activation function $\phi$ satisfying $\inf_{x \in \R} \phi'(x) = 0$ and $\sup_{x \in \R} \phi'(x) = 1$,
\begin{enumerate}
\item \label{fact:nnell2:osLFR_alpha(W)>0_apx}
if $\lmax >0$, then $$\displaystyle \osLip_{2,\ql{\lmax}}(\ffr) = -1{+}\lmax,$$ with $\ql{\lmax}\in\R^{n\times n}$ defined in~\eqref{eq:ql};
\item \label{fact:nnell2:osLFR_alpha(W)<0_apx}
if $\lmax < 0$, then $$\displaystyle \osLip_{2,(-W)^{1/2}}(\ffr) = -1.$$
\end{enumerate}
\end{lem}
\begin{proof}
The proof of both parts follows by applying Theorem~\eqref{thm:ell2_osLip_fr} and noticing that under the above assumptions for any log-norm $\mu$ it holds the reverse inequality
\beq
\label{eq:rem}
\lognorm{}{D{\ffr}(x)} \geq -1{+}\lmax.
\eeq
To prove~\eqref{eq:rem}, let $\map{h}{\R \setminus \nondiff}{[0,1]}$ be the function defined by $h(x) = \phi'(x)$ where $\nondiff$ is the measure zero set of points in $\R$ where $\phi$ is not differentiable. It is well-known that for any closed and bounded set $S \subset \R$, $S \supseteq \{\inf(S), \sup(S)\}$. Then, since $h$ is bounded, the closure of $\Img(h)$ satisfies 
\begin{equation}\label{eq:closure-containment}
    \overline{\Img(h)} \supseteq \Big\{\inf_{x \in \R \setminus \nondiff} \phi'(x), \sup_{x \in \R \setminus \nondiff} \phi'(x)\Big\} = \{0,1\}.
\end{equation}
Letting $\nondiffBig$ be the measure zero points in $\R^n$ where $\Phi$ is not differentiable, we compute
\begin{align}
\sup_{x \in \R^n\setminus \nondiffBig} \lognorm{}{D\Phi(Wx+u)W} &= \sup_{x \in \R^n \setminus \nondiffBig} \lognorm{}{D\Phi(x)W}\label{eq:invertibility}\\
&=\sup\setdef{\lognorm{}{[d]W}}{d_i \in \Img(h), \forall i} \\
&=\max\setdef{\lognorm{}{[d]W}}{d_i \in \overline{\Img(h)}, \forall i} \\
&\geq \max_{d \in \{0,1\}^n} \lognorm{}{[d]W}\label{eq:max-closure}\\
& = \max_{d \in [0,1]^n} \lognorm{}{[d]W}.\label{eq:mu-convexity}
\end{align}
We justify the above (in)equalities as follows. Equality~\eqref{eq:invertibility} holds because $W$ is invertible. Inequality~\eqref{eq:max-closure} holds because of the condition~\eqref{eq:closure-containment}. Finally, equality~\eqref{eq:mu-convexity} follows because $\mu$ is a convex function of its argument and the maximum value of a convex function over a polytope occurs at one of its vertices. 

In particular, for the respective choice of norm in parts \ref{fact:nnell2:osLFR_alpha(W)>0_apx} and~\ref{fact:nnell2:osLFR_alpha(W)<0_apx}, the result is proved in view of Theorem~\ref{thm:prop_lognorm} and the translation property for log-norms.
\end{proof}
\bibliographystyle{plainurl+isbn}
\bibliography{alias,Main,FB}
\end{document}